\documentclass [12pt,a4paper]{amsart}
\usepackage{amsfonts}
\usepackage{amsthm}
\usepackage{amsmath, pb-diagram}
\usepackage{amscd}
\usepackage[latin2]{inputenc}
\usepackage{t1enc}
\usepackage[mathscr]{eucal}
\usepackage{indentfirst}
\usepackage{graphicx}
\usepackage{graphics,epsfig}
\usepackage{pict2e}
\usepackage{epic}
\numberwithin{equation}{section}
\usepackage[margin=2.9cm]{geometry}
\usepackage{epstopdf}
\usepackage{xfrac}
\usepackage{float}
\usepackage[shortlabels]{enumitem}
\usepackage{mathtools}
\usepackage{amsmath,gauss,tikz}
\usepackage{lscape}
\usepackage{array}
\usepackage{setspace}
\usepackage{enumitem}

\DeclareMathOperator{\cat}{cat}

\allowdisplaybreaks

\theoremstyle{plain}
\newtheorem{theo}{Theorem}[section]
\newtheorem{cor}[theo]{Corollary}

\newtheorem{prop}[theo]{Proposition}

\theoremstyle{definition}
\newtheorem{defn}[theo]{Definition}

\newtheorem{rem}[theo]{Remark}

\begin{document}
	
\title{On the LS-category and topological complexity of projective product spaces}

\author{Seher F\.{I}\c{S}EKC\.{I}}
\address{Ege University, Izmir, Turkey}
\author{Lucile Vandembroucq}
\address{Centro de Matem\'atica, Universidade do Minho, Campus de Gualtar, 4710-057 Braga, Portugal.
}
\thanks{This work has been realized during a stay of the first author at the Centre of Mathematics of the University of Minho during the period October 2019-July 2020. The first author is granted by a fellowship by the Scientific and Technological Research Council of Turkey TUBITAK-2211/A and supported  by the Scientific and Technological Research Council of Turkey International Doctoral Research Fellowship Programme TUBITAK-2214/A (Project Number:1059B141801086). The research of the second author was partially supported by Portuguese Funds through FCT -- Funda\c c\~ao para a Ci\^encia e a Tecnologia, within the Project UID/MAT/00013/2020.}

\date{\today}
\maketitle 

\begin{abstract}
	We determine the Lusternik-Schnirelmann category of the projective product spaces introduced by D. Davis. We also obtained an upper bound for the topological complexity of these spaces, which improves the estimate given by J. González, M. Grant, E. Torres-Giese, and M. Xicoténcatl.
\end{abstract}

\section{Introduction}
We let $\bar{n}$ symbolize an $r$-tuple $(n_1 , \ldots , n_r)$ of positive integers with $n_1 \leq \ldots \leq n_r$. We consider the product $S_{\bar{n}}: = S^{n_1} \times \cdots \times S^{n_r}$ and, given $x_i \in S^{n_i}$, we write $\bar{x} = (x_1, \ldots, x_r)$ for the corresponding element of $S_{\bar{n}}$. The quotient space $$P_{\bar{n}}: = S_{\bar{n}}/(\bar x \sim-\bar x)=(S^{n_1} \times \cdots \times S^{n_r}) / (({x_1, \ldots, x_r}) \sim ({- x_1, \ldots, - x_r}))$$ with respect to the diagonal action of $\mathbb{Z}_2$ on $S_{\bar{n}}$ has been introduced by D. Davis \cite{Davis} and is called \textit{projective product space}. This is a manifold of dimension $dim(P_{\bar{n}}) = dim(S_{\bar{n}}) = \sum n_i$ and, when $r = 1$, the space $P_{\bar{n}}$ coincides with the usual real projective space $P^{n_1}$.
Recently, the notion of projective product space was used in \cite{CGGGL} and generalized in \cite{SZ}.  

In this note, we study the (normalized) Lusternik-Schnirelmann category (cat) and Farber's topological complexity (TC) of the space $P_{\bar n}$. This study has been initiated by J. González, M. Grant, E. Torres-Giese, M. Xicoténcatl in \cite{GGTX}, where the following result is established:

\begin{theo}\cite[Theorem 3.8]{GGTX} \label{GGTX-mainThm} Let $k$ represent the number of spheres $S^{n_q}$ with $n_q$ even and $q>1$. Then $TC(P_{\bar{n}}) < (TC(P^{n_1}) + 1)(r + k)$.
\end{theo}

In this paper, we first determine the LS-category of $P_{\bar n}$: 


\begin{theo} \label{theo-cat} $\text{cat}(P_{\bar{n}})=\text{cat}(P^{n_1})+r-1=n_1+r-1$.
\end{theo}

This is obtained through the construction of an explicit categorical cover together with the knowledge of the  cohomology of $P_{\bar n}$ over ${\mathbb Z}_2$, which has been determined by Davis \cite[Theorem 2.1]{Davis} (see Theorem \ref{Davis-cohomology} below).\\

For the topological complexity, we establish:

\begin{theo} \label{theo-TC}
	$\text{TC}(P_{\bar{n}}) \leq \text{TC}(P^{n_1}) + \sum_{q=2}^{r}\text{TC}(S^{n_q}).$
\end{theo}

In terms of the number $k$ of spheres $S^{n_q}$ with $n_q$ even and $q>1$, the upper bound of Theorem \ref{theo-TC} can be written $ \text{TC}(P^{n_1}) + r+k-1$, which permits us to see that Theorem \ref{theo-TC} improves Theorem \ref{GGTX-mainThm}. As mentionned in \cite{GGTX}, using Davis' description of $H^*(P_{\bar n}; {\mathbb Z}_2))$, the zero-divisor-cup-length of $P_{\bar n}$ over $\mathbb{Z}_2$ can be expressed as $\text{zcl}_{{\mathbb Z}_2}(P^{n_1})+r-1$. We then obtain:
	
 \begin{cor} \label{cor-TC} If $\textrm{zcl}_{{\mathbb Z}_2}(P^{n_1})=\text{TC}(P^{n_1})$ and $n_q$ is odd for $q>1$ then $\textrm{TC}(P_{\bar{n}}) = \textrm{TC}(P^{n_1}) + \sum_{q=2}^{r}\textrm{TC}(S^{n_q})= \textrm{TC}(P^{n_1})+r-1.$
 \end{cor}

The upper bound in Theorem \ref{theo-TC} is obtained through the construction of an explicit motion planner which uses the characterization of $\text{TC}(P^{n_1})$ in terms of non-singular maps due to M. Farber, S. Yuzvinsky and S. Tabachnikov \cite{FTY}. We note that, using a strong version of non-singular maps, an explicit motion planner for polyhedral products of real projective spaces has recently been constructed in \cite{AGO}.

\section{Lusternik-Schnirelmann category of $P_{\bar n}$}

As $P_{\bar n}$ is a finite, path-connected CW complex, we give all the useful definitions and characterizations for such spaces.

\begin{defn} Let $X$ be a finite, path-connected CW complex. A subset $A \subset X$ is called categorical if the inclusion $i: A \hookrightarrow X$ is nullhomotopic. The Lusternik-Schnirelman category $cat(X)$ of $X$ is defined as the least integer $k$ that admits a cover of $X$ by $k + 1$ open categorical subsets $U_0, \ldots , U_k \subset X$. 
\end{defn}

Let $R$ be a commutative unitary ring. Recall that the cup-length over $R$ of a path-connected space $X$, ${\rm cuplength}(X) = {\rm cuplength}_R(X)$, is the longest length $k$ of a nonzero product $c_1 \smile \cdots \smile c_k \neq 0$ of cohomology classes $c_1, \cdots, c_k \in H^+(X; R)$ and provides a lower bound for the Lusternik-Schnirelmann category of $X$, ${\rm cuplength}(X) \leq cat(X)$.

As is well-known, the cup-length over $\mathbb{Z}_2$ of the $m$-dimensional projective space $P^m$ is equal to $m$. The ${\mathbb Z}_2$ cohomology of $P_{\bar n}$ has been determined by Davis:  

\begin{theo}\label{Davis-cohomology} (\cite[Theorem 2.1]{Davis}) Let $\bar{n} = (n_1 , \ldots , n_r)$ such that $n_1 \leq \ldots \leq n_r$. If $n_1 < n_2$ or $n_1$ is odd, the mod 2 cohomology ring of $P_{\bar{n}}$ is given by $$H^*(P_{\bar{n}};\mathbb{Z}_2) = H^*(P^{n_1};\mathbb{Z}_2) \otimes \Lambda[a_{2},\ldots,a_{r}]$$ where $dim(a_1)=1$, $dim(a_i)=n_i$ for $i>1$, and $\Lambda$ denotes an exterior algebra. If $n_1$ is even and $n_1=\cdots=n_k<n_{k+1}$ for some $k>1$, $H^*(P_{\bar{n}};\mathbb{Z}_2)$ is the same as above with the extra relation given by  $a_{i}^2=a^{n_1}a_{i}$ for $2\leq i\leq k$.
\end{theo}

By considering  the longest nonzero product $a_1^{n_1} \smile a_2\cdots \smile a_r \neq 0$ in $H^*(P_{\bar n};{\mathbb Z}_2)$ we obtain:

\begin{prop}\label{cuplength} ${\rm cuplength}_{{\mathbb Z}_2}(P_{\bar n})={\rm cuplength}_{{\mathbb Z}_2}(P^{n_1})+r-1=n_1+r-1$.
	\end{prop}

In order to prove Theorem \ref{theo-cat}, we will construct an explicit cover of $P_{\bar n}$ with contractible subsets using the following characterization of the category:

\begin{prop}\cite[Lemma 1.35]{CLOT}\label{Caracterization-cat} Suppose that $X$ is a path-connected finite CW-complex. We have $\text{cat}(X)\leq k$ if and only if there exists an increasing sequence of open sets
	\begin{displaymath}
	\emptyset  = U_{-1}\subset U_0 \subset U_1 \subset \ldots \subset U_k = X 
	\end{displaymath}
	such that, for any $i\in \{0,\cdots,k\}$, $U_i -U_{i-1}$ is contractible in $X$.
\end{prop}

\begin{rem} The sets $F_i=U_i -U_{i-1}$ provide a cover of $X$ by $k+1$ disjoint subsets which are contractible in $X$. We note that, in \cite[Lemma 1.35]{CLOT}, $X$ is not supposed to be a finite CW-complex but it is required that each $F_i=U_i -U_{i-1}$ is contained in an open set of $X$ which is contractible in $X$. Assuming that $X$ is a finite CW-complex (and therefore an ENR space) and adapting the proof of \cite[Prop. 4.12]{Far2}, we can just ask that $F_i$ is contractible in $X$. Actually, when $X$ is a metric ANR space, we know by \cite{Srinivasan} (see also \cite{Calcines}), that $\cat(X)\leq k$ if and only there exists a cover by $k+1$ subsets contractible in $X$ without further condition on the subsets.  
\end{rem}

\begin{proof}[Proof of Theorem \ref{theo-cat}] 
	We first fix some general notation. Writing $$S^{m}=\{(u_0 , \ldots , u_k, \ldots, u_{m}) \in S^m \subset \mathbb{R}^{m+1}~|~ \sum u_k^2=1\},$$ we denote by $p_j:S^{m}\to {\mathbb R}$ the obvious projection ($j=0,\cdots,m$) and by $a_j$ the unique point of $S^m$ such that $p_j(a_j)=1$. Note that, for any $x\in S^m$, $p_j(-x)=-p_j(x)$. When $j=m$, we will use the special notation $A_m:=a_m=(0,\cdots,0,1)$ and we fix a meridian path $\mu_0(A_m,-A_m):I\to S^m$ such that $\mu_0(0)=A_m$ and $\mu_0(1)=-A_m$. We will denote by $\mu_0(-A_m,A_m)$ the path given by $\mu_0(-A_m,A_m)(t)=-\mu_0(A_m,-A_m)(t)$, that is $\mu_0(-A_m,A_m)=-\mu_0(A_m,-A_m)$.
	For non-antipodal points $A,B$ $(A \neq -B)$ of $S^m$, let $\lambda(A,B): I \to S^m$ be the geodesic path from $A$ to $B$. Note that $\lambda(-A , -B) = -\lambda(A,B)$ and that $\lambda(A,A)$ is the constant path.\\

We now define a cover of $S^{n_1}$ which induces a cover of $P^{n_1}$ by categorical subsets. In our constructions we use a formalism which was inspired by \cite{CP}.

For a subset $L \subset \{ 0,1, \ldots , n_1 \}$, let $|L|$ denote the cardinality of $L$ and consider
\begin{displaymath}
S_L = \left\{ x \in S^{n_1} \text{ : } p_l(x) \neq 0 \text{ if } l \in L \right\}.
\end{displaymath}
By setting $$U_{i} = \bigcup_{|L| = (n_1+1) - i} S_L$$ for $0 \leq i \leq n_1$ and $U_{-1} = \emptyset$, we have an increasing sequence of open subsets of $S^{n_1}$:
\begin{displaymath}
\emptyset  = U_{-1} \subset U_0 \subset \ldots \subset U_{n_1} = S^{n_1}
\end{displaymath}
Note that $U_{n_1}=S^{n_1}$ since the projections $p_0$,...,$p_{n_1}$ cannot vanish all at the same time.

Considering, for $L \subset \{ 0,1, \ldots , n_1 \}$,
\begin{displaymath}
Q_L  = \left\{ x \in S^{n_1} \text{ : } p_l(x) \neq 0 \text{ if } l \in L \text{ and } p_l(x) = 0 \text{ if } l \notin L \right\}
\end{displaymath}
we have, for $0 \leq i \leq n_1$, $$F_{i}:=U_{i} - U_{i-1} =  \bigcup_{|L| = (n_1+1) - i} Q_L.$$
Note that all the sets $S_L$, $Q_L$, $U_i$, $F_i$ are saturated with respect to the antipodal relation $x\sim -x$ on $S^{n_1}$.
Note also that, for $L_1 = L_2$ with same cardinality $|L_1| = |L_2|$ the sets $Q_{L_1}$ and $Q_{L_2}$ are topologically disjoint and that $\bigcup_{i=0} ^{n_1} F_{i}$ is a cover of $S^{n_1}$.

For $L \subset \{ 0,1, \ldots , n_1 \}$, we consider the map $\psi_L: Q_L \to (P^{n_1})^I$
\begin{displaymath}
\psi_L (x) = \begin{cases} 
[\lambda(x,a_{l_0})], &\text{ if } p_{l_0} (x) > 0, \\
[\lambda(x,-a_{l_0})], &\text{ if } p_{l_0} (x) < 0
\end{cases}.
\end{displaymath}
where $l_0 = \text{min }L$. The map is continuous on $Q_L$ and satisfies $\psi_L(x)(0)=[x]$, $\psi_L(x)(1)=[a_{l_0}]$, and $\psi_L(x)=\psi_L(-x)$. This shows that the subset $Q_L/\sim$ is contractible in $P^{n_1}$. As the sets $Q_L$ with $|L|$ constant are topologically disjoints, we obtain that, for $0 \leq i \leq n_1$, $F_i/\!\sim$ is a union of topologically disjoint subsets which are contractible in $P^{n_1}$. Since $P^{n_1}$ is path-connected, we can conclude that $F_i/\!\sim$ is itself contractible in $P^{n_1}$.

For $2 \leq q \leq r$, we consider the increasing sequence of open sets
\[V_{-1}=\emptyset \quad V_0=\{x\in S^{n_q}~|~x\neq \pm A_{n_q}\} \quad V_1=S^{n_q}\]
as well as the sets $G_0:=V_0-V_{-1}=V_0$ and $G_1:=V_1-V_0=\{\pm A_{n_q}\}$.
These sets are saturated with respect to the antipodal relation.

Let $L \subset\{ 0,1, \ldots , n_1 \}$ with  $l_0 = \text{min }L$ and, for $2 \leq q \leq r$, let $j_q \in \{ 0,1 \}$. We define 
\[\psi_{(L,j_2 , \ldots , j_r)} : Q_L \times \Pi_{q=2} ^r G_{j_q} \to P_{\bar n}^I\]
by
\begin{displaymath}
\psi_{(L,j_2 , \ldots , j_r)} (x_1,x_2,\ldots ,x_r) = \begin{cases}
[\lambda(x_1,a_{l_0}) , y_2 , \ldots , y_r], &  \text{ if } p_{l_0} (x_1) > 0 \\
[\lambda(x_1,-a_{l_0}) , z_2 , \ldots , z_r], &  \text{ if } p_{l_0} (x_1) < 0
\end{cases}
\end{displaymath}
where

$y_q = \begin{cases}
\mu_0(-A_{n_q},A_{n_q}), & \text{ if } x_q = -A_{n_q} \\
\lambda (x_q , A_{n_q}), & \text{ if } x_q \neq -A_{n_q} \\
\end{cases}$

and 

$z_q = \begin{cases}
\mu_0(A_{n_q},-A_{n_q}), & \text{ if } x_q = A_{n_q} \\
\lambda (x_q , -A_{n_q}), & \text{ if } x_q \neq A_{n_q} \\
\end{cases}$

This map is continuous, well-defined on $Q_L \times \Pi_{q=2} ^r G_{j_q} $ and satisfies
\[\psi_{(L,j_2 , \ldots , j_r)} (\bar x)(0)=[\bar x] \quad 
\psi_{(L,j_2 , \ldots , j_r)} (\bar x)(0)=[a_{l_0},A_{n_2},\ldots ,A_{n_r}]\]
Moreover, we can check that the map is compatible with the diagonal antipodal relation and hence permits us to see that $(Q_L \times \Pi_{q=2} ^r G_{j_q})/\sim $ is a subset contractible in $P_{\bar n}$.

For $i \in \{ 0,1, \ldots , n_1 \}$ and $j_q \in \{0,1\}$ for each $2 \leq q \leq r$, we built 
\begin{center}
	$W_s = \bigcup_{i+\sum_{q=2}^r j_q =s} U_i \times \prod_{q=2}^r  V_{j_q} \subset S^{n_1} \times S^{n_2} \ldots \times S^{n_r}	$
\end{center}
where $s=0, \ldots ,n_1+r-1$. Therefore, there exists a tower of open subsets
\begin{displaymath}
\emptyset = W_{-1} \subset W_0 \subset \ldots \subset W_{n_1+r-1} = S^{n_1} \times S^{n_2} \ldots \times S^{n_r}
\end{displaymath}
We have
\begin{displaymath}
W_{s} - W_{s-1} = \bigcup_{i+\sum_{q=2}^r j_q =s}(U_{i}-U_{i-1}) \times \prod_{q=2}^r (V_{j_q} - V_{j_q-1}) = \bigcup_{i+\sum_{q=2}^r j_q =s}F_i \times \prod_{q=2}^r G_{j_q}
\end{displaymath}
which is a topologically disjoint union and each $F_i \times \prod_{q=2}^r G_{j_q}$ is itself a topologically disjoint union of $(Q_L \times \Pi_{q=2} ^r G_{j_q})/\sim $ with $|L|$ constant. As already mentioned, all the subsets are saturated with respect to the diagonal antipodal relation. Passing to the quotient we get 
a tower of open subsets of $P_{\bar n}$ 
\begin{displaymath}
\emptyset = \tilde{W}_{-1} \subset \tilde{W}_0 \subset \ldots \subset \tilde{W}_{n_1+r-1)} =  P_{\bar n}
\end{displaymath}
and 
\begin{displaymath}
\tilde{W}_{s} - \tilde{W}_{s-1}
= \bigcup_{i+\sum_{q=2}^r j_q =s} \dfrac {F_i \times \prod_{q=2}^r G_{j_q}} {\sim}
\end{displaymath}
is a topologically disjoint union. As $({F_i \times \prod_{q=2}^r G_{j_q}})/\!{\sim}$ is itself a topologically disjoint union of $(Q_L \times \Pi_{q=2} ^r G_{j_q})/\!\sim $ and as these spaces are contractible in $P_{\bar n}$, which is path-connected, we can conclude that each $\tilde{W}_{s} - \tilde{W}_{s-1}$ is contractible in $P_{\bar n}$. Therefore by Proposition \ref{Caracterization-cat}, we obtain that  
$\text{cat}(P_{\bar{n}}) \leq n_1+r-1$. By Proposition \ref{cuplength}, we conclude that $\text{cat}(P_{\bar{n}}) = n_1+r-1$.
\end{proof}

\section{Topological complexity of $P_{\bar n}$}

\begin{defn}  Let $X$ be a finite, path-connected CW complex. The (normalized) topolological complexity of $X$, $TC(X)$, is the least integer $k$ such that there exists a cover of $X \times X$ by $k + 1$ open subsets $U_0, U_1, . . . , U_k \subset X \times X$ on each of which the fibration 
\begin{center}
$ev_{0,1}: X^I \rightarrow X \times X$, $\gamma \mapsto (\gamma(0), \gamma(1))$,
\end{center}
 admits a continuous section. 
\end{defn}

%
%


%

In order to establish Theorem \ref{theo-TC}, we will use the following characterization:

\begin{prop}\cite[Proposition 4.12]{Far2} \label{Caracterization-TC}  Let $X$ be a finite, path-connected CW complex. We have $\text{TC}(X)\leq k$ if and only if there exists  $s:X\times X \to X^I$ and an increasing sequence of open sets
\begin{displaymath}
\emptyset  = U_{-1}\subset U_0 \subset U_1 \subset \ldots \subset U_k = X \times X
\end{displaymath}
such that $ev_{0,1}\circ s=id$ and, for any $i\in \{0,\cdots,k\}$, $s$ is continuous on $U_i -U_{i-1}$.
\end{prop}

\begin{rem} The sets $F_i=U_i -U_{i-1}$ provide a cover of $X\times X$ by $k+1$ disjoint subsets on each of which there is a continuous section of $ev_{0,1}$. This defines a motion planner for $X$.
\end{rem}

\begin{proof}[Proof of Theorem \ref{theo-TC}]
	Through the obvious homeomorphism we think of  $P_{\bar{n}} \times P_{\bar{n}}=S_{\bar{n}} \times S_{\bar{n}} / \sim$ as the quotient of $(S^{n_1}\times S^{n_1})\times \cdots \times (S^{n_r}\times S^{n_r})$ with respect to the relation
	
\begin{equation}\label{relation}
(x_1,y_1,\dots,x_r,y_r)\sim (x'_1,y'_1,\dots,x'_r,y'_r) \Leftrightarrow \left\{\begin{array}{rcc} 
\forall i& x_i=x'_i \text{ and } y_i=y'_i \\
\text{or } \forall i& x_i=-x'_i \text{ and } y_i=y'_i \\
\text{or } \forall i& x_i=x'_i \text{ and } y_i=-y'_i \\
\text{or } \forall i& x_i=-x'_i \text{ and } y_i=-y'_i \\
\end{array}\right.
\end{equation}	
	
We first recall from \cite{FTY} and \cite{Far} the construction of motion planners for the real projective space $P^{n_1}$ and for a sphere $S^{n_q}$. We will next see how to assemble them in order to obtain a motion planner for $P_{\bar n}$.

	Suppose that $\text{TC}(P^{n_1})=k$. As proven in \cite{FTY}, there exists a non-singular map $f=(f_0,\cdots,f_k) : \mathbb{R}^{n_1 + 1} \times \mathbb{R}^{n_1 + 1} \to \mathbb{R}^{k+1}$. The $k+1$ scalar maps $f_0 , f_1 , \ldots , f_{k}: \mathbb{R}^{n_1 + 1} \times \mathbb{R}^{n_1 + 1} \to \mathbb{R}$ satisfy $f_i(ax_1,by_1)=abf_i(x_1,y_1)$ for $(x_1,y_1) \in S^{n_1} \times S^{n_1}$ and $a,b \in \mathbb{R}$ and do not vanish all at the same time (except in $(0,0)$). Moreover, according to \cite[Lemma 11]{FTY}, we can assume that $f_0(x_1,x_1)>0$ for any $x_1 \in S^{n_1}$.
	
For a subset $L \subset \{ 0,1, \ldots , k \}$, let
\begin{displaymath}
S_L = \left\{ (x_1,y_1) \in S^{n_1} \times S^{n_1} \text{ : } f_l(x_1,y_1) \neq 0 \text{ if } l \in L \right\}.
\end{displaymath}
 
By setting $$U_{i} = \bigcup_{|L| = (k+1) - i} S_L$$ for $0 \leq i \leq k$ and $U_{-1} = \emptyset$, we have an increasing sequence of open subsets of $S^{n_1} \times S^{n_1}$:
\begin{displaymath}
\emptyset  = U_{-1} \subset U_0 \subset \ldots \subset U_k = S^{n_1} \times S^{n_1}
\end{displaymath}
Note that $U_k=S^{n_1}\times S^{n_1}$ since the $k+1$ scalar functions $f_0$,...,$f_k$ cannot vanish all at the same time.

Considering, for $L \subset \{ 0,1, \ldots , k \}$,
\begin{displaymath}
Q_L = \left\{ (x_1,y_1) \in S^{n_1} \times S^{n_1} \text{ : } f_l(x_1,y_1) \neq 0 \text{ if } l \in L \text{ and } f_l(x_1,y_1) = 0 \text{ if } l \notin L \right\}
\end{displaymath}
we have, for $0 \leq i \leq k$, $$F_{i}:=U_{i} - U_{i-1} =  \bigcup_{|L| = (k+1) - i} Q_L .$$
Note that all the sets $S_L$, $Q_L$, $U_i$, $F_i$ are saturated with respect to the equivalence relation on $S^{n_1}\times S^{n_1}$ induced by the antipodal relation $x_1\sim -x_1$ on $S^{n_1}$.
Note also that, for $L_1 = L_2$ with same cardinality $|L_1| = |L_2|$ the sets $S_{L_1}$ and $S_{L_2}$ are topologically disjoint and that $\bigcup_{i=0} ^{k} F_i$ is a cover of $S^{n_1} \times S^{n_1}$.

As before, for non-antipodal points $A,B$ $(A \neq -B)$ of a sphere $S^m$, we denote by $\lambda(A,B): I \to S^m$ the geodesic path from $A,B$. Notice that $\lambda(-A,-B) = -\lambda(A,B)$ and that $\lambda(A,A)$ is the constant path.

For $L \subset \{ 0,1, \ldots , k \}$, we consider the map $\psi_L: Q_L \to (P^{n_1})^I$
\begin{displaymath}
\psi_L (x_1,y_1) = \begin{cases} 
[\lambda(x_1,y_1)], &\text{ if } f_{l_0} (x_1,y_1) > 0, \\
[\lambda(-x_1,y_1)], &\text{ if } f_{l_0} (x_1,y_1) < 0
\end{cases}.
\end{displaymath}
where $l_0 = \text{min }L$. Recall that, for any $x_1\in S^{n_1}$, $f_0(x_1,x_1)=f_0(-x_1,-x_1>0$ and, consequently $f_0(x_1,-x_1)=f_0(-x_1,x_1)<0$. Therefore, if $(\pm x_1,\pm x_1)\in Q_L$, then $l_0 = \text{min }L=0$. This ensures that $\psi_L$ is well-defined on pairs of antipodal points. 
The map is continuous on $Q_L$ and satisfies $\psi_L(x_1,y_1)=\psi_L(\pm x_1,\pm y_1)$. As the sets $Q_L$ with $|L|$ constant are topologically disjoints, we obtain, for $0 \leq i \leq k$, a continuous map $\psi_i : F_i \to (P^{n_1})^I$ by setting $\psi_i | Q_L = \psi_L$. This map satisfies $\psi_i(x_1,y_1)=\psi_i(\pm x_1,\pm y_1)$ and the induced map $\bar{\psi}_i : F_i /\!\sim \,\to (P^{n_1})^I$ gives us an explicit motion planner on $P^{n_1}$ which essentially corresponds to the one described in \cite{FTY}.

For $2 \leq q \leq r$, we will use the following increasing sequence of open subsets of $S^{n_q} \times S^{n_q}$ together with the associated complements:

When $n_q$ is odd, we consider:
\begin{flalign*}
V_{-1} &= \emptyset \\
V_0 &= \left\{ (x_q, y_q) \in S^{n_q} \times S^{n_q} \text{ : } y_q \neq \pm x_q \right\} \\
V_1 &= S^{n_q} \times S^{n_q}. 
\end{flalign*}
and
\begin{flalign*}
G_0 &= V_0 - V_{-1} =  \left\{ (x_q, y_q) \in S^{n_q} \times S^{n_q} \text{ : } y_q \neq \pm x_q \right\} \\
G_1 &= V_1 - V_0 = \left\{ (x_q, y_q) \in S^{n_q} \times S^{n_q} \text{ : } y_q = \pm x_q \right\}.
\end{flalign*}

When $n_q$ is even, we consider:
\begin{flalign*}
V_{-1} &= \emptyset \\
V_0 &= \left\{ (x_q, y_q) \in S^{n_q} \times S^{n_q} \text{ : } y_q \neq \pm x_q \right\} \\
V_1 &= S^{n_q} \times S^{n_q} \setminus \left\{ (x_q, y_q) \in S^{n_q} \times S^{n_q} \text{ : } y_q = \pm x_q \text{ and } x_q = \pm A_{n_q} \right\} \\
V_2 &= S^{n_q} \times S^{n_q}.
\end{flalign*}
and
\begin{flalign*}
G_0 &= V_0 - V_{-1} =  \left\{ (x_q, y_q) \in S^{n_q} \times S^{n_q} \text{ : } y_q \neq \pm x_q \right\} \\
G_1 &= V_1 - V_0 =  \left\{ (x_q, y_q) \in S^{n_q} \times S^{n_q} \text{ : } y_q = \pm x_q , x_q \neq \pm A_{n_q} \right\}.\\
G_2& = V_2 - V_1 =  \left\{ (x_q, y_q) \in S^{n_q} \times S^{n_q} \text{ : } y_q = \pm x_q , x_q = \pm A_{n_q} \right\}.
\end{flalign*}
Here, as before, $A_{n_q}=(0,\cdots,0,1)\in S^{n_q}$.

Recall that the classical motion planner for a sphere can be given as follows:
\begin{itemize}
	\item For $(x_q, y_q) \in G_0$ and for $(x_q,x_q)\in G_1 \cup G_2$, we consider the geodesic path $\lambda(x_q, y_q)$.
\item For $(x_q , -x_q) \in G_1$, we consider the geodesic meridian $\mu(x_q , -x_q)$ from $x_q$ to $-x_q$ in the direction of $\chi(x_q)$. Here $\chi$ is the symmetric tangent vector field on $S^{n_q}\subset {\mathbb R}^{n_q+1}$ given by $\chi(u_1,v_1,\cdots, u_m,v_m)=(-v_1,u_1,\cdots,-v_m,u_m)$ if $n_q=2m-1$ is odd and by $\chi(u_1,v_1,\cdots, u_m,v_m,u_{m+1})=(-v_1,u_1,\cdots,-v_m,u_m,0)$ if $n_q=2m$ is even. Note that  $\mu(-x_q , x_q)=-\mu(x_q,-x_q)$. 
\item For $(A_{n_q},-A_{n_q})$, we fix a meridian $\mu_0(A_{n_q} , -A_{n_q})$ from $A_{n_q}$ to $-A_{n_q}$ and we set $\mu_0(-A_{n_q},A_{n_q}) = -\mu_0 (A_{n_q} , -A_{n_q})$.
\end{itemize}

We first assemble these motion planners in the following way.\\

Let $L \subset\{ 0,1, \ldots , k \}$ with  $l_0 = \text{min }L$ and, for $2 \leq q \leq r$, let $j_q \in \{ 0,1 \}$ when $n_q$ is odd or $j_q \in \{ 0,1,2 \}$ when $n_q$ is even. We define 
\[\psi_{(L,j_2 , \ldots , j_r)} : Q_L \times \Pi_{q=2} ^r G_{j_q} \to P_{\bar n}^I\]
by
\begin{displaymath}
\psi_{(L,j_2 , \ldots , j_r)} (x_1,y_1,x_2 , y_2 , \ldots , x_r , y_r) = \begin{cases}
[\lambda(x_1,y_1) , z_2 , \ldots , z_r], & f_{l_0} (x_1,y_1) > 0 \\
[\lambda(-x_1,y_1) , z'_2 , \ldots , z'_r], & f_{l_0} (x_1,y_1) < 0
\end{cases}
\end{displaymath}
where, for $n_q$ odd,

$z_q = \begin{cases}
\mu(x_q, y_q), & \text{ if } y_q= -x_q \\
\lambda (x_q, y_q), & \text{ otherwise } (y_q\neq \pm x_q \text{ or } y_q= x_q)
\end{cases}$

$z'_q = \begin{cases}
\mu(-x_q, y_q), & \text{ if } y_q= x_q \\
\lambda (-x_q, y_q), & \text{ otherwise } (y_q\neq \pm x_q \text{ or } y_q= -x_q)
\end{cases}$

and, for $n_q$ even,

$z_q = \begin{cases}
\mu (x_q, y_q), & \text{ if } y_q= -x_q , x_q \neq \pm A_{n_q} \\
\mu_0 (x_q, y_q), & \text{ if } y_q= -x_q , x_q = \pm A_{n_q} \\
\lambda (x_q, y_q), & \text{ otherwise } (y_q\neq \pm x_q \text{ or } y_q= x_q)
\end{cases}$

$z'_q = \begin{cases}
\mu (-x_q, y_q), & \text{ if } y_q = x_q , x_q \neq \pm A_{n_q} \\
\mu_0 (-x_q, y_q), & \text{ if } y_q = x_q , x_q = \pm A_{n_q} \\
\lambda (-x_q, y_q), & \text{ otherwise } (y_q \neq \pm x_q \text{ or } y_q = -x_q)
\end{cases}$

This map is continuous and well-defined on $Q_L \times \Pi_{q=2} ^r G_{j_q}$. Moreover, this space is saturated with respect to the relation (\ref{relation}) and we can check that the map is compatible with this relation.

For $i \in \{0,...,k\}$, we now define $$\psi_{(i , j_2 , \ldots , j_r)}: F_{i} \times \Pi_{q=2} ^r G_{j_q} \to P_{\bar n}^I$$
by setting $\psi_{(i , j_2 , \ldots , j_r)} \big\vert _{Q_L \times \Pi_{q=2} ^r G_{j_q}} = \psi_{(L , j_2 , \ldots , j_r)}$. Passing to the quotient, we get a continuous map 
\begin{displaymath}
\tilde{\psi}_{(i , j_2 , \ldots , j_r)}: (F_i \times \Pi_{q=2} ^r G_{j_q}) / \sim  \to P_{\bar n}^I.
\end{displaymath}
.

For $i \in \{ 0,1, \ldots , k \}$, $j_q \in \{0,1\}$ when $n_q$ is odd or $j_q \in \{0,1,2\}$ when $n_q$ is even ($2 \leq q \leq r$) we built 
\begin{center}
$W_s = \bigcup_{i+\sum_{q=2}^r j_q =s} U_i \times \prod_{q=2}^r  V_{j_q} \subset  (S^{n_1} \times S^{n_1}) \times (S^{n_2} \times S^{n_2}) \times \ldots \times (S^{n_r} \times S^{n_r}) 
$
\end{center}
where $s=0, \ldots ,k+ \sum_{q=2}^rTC(S^{n_q})=TC(P^{n_1}) + \sum_{q=2}^rTC(S^{n_q})$. Therefore, there exists a tower of open subsets
\begin{displaymath}
\emptyset = W_{-1} \subset W_0 \subset \ldots \subset W_{k + \sum_{q=2}^r TC(S^{n_q})} = (S^{n_1} \times S^{n_1}) \times (S^{n_2} \times S^{n_2}) \times \ldots \times (S^{n_r} \times S^{n_r}). 
\end{displaymath}
We have
\begin{displaymath}
W_{s} - W_{s-1} = \bigcup_{i+\sum_{q=2}^r j_q =s}(U_{i}-U_{i-1}) \times \prod_{q=2}^r (V_{j_q} - V_{j_q-1}) = \bigcup_{i+\sum_{q=2}^r j_q =s}F_{i} \times \prod_{q=2}^r G_{j_q}
\end{displaymath}
which is a topologically disjoint union. As already mentioned, all the subsets are saturated subsets of
$S^{n_1} \times S^{n_1} \times S^{n_2} \times S^{n_2} \times \ldots \times S^{n_r} \times S^{n_r} \cong S_{\bar n}\times S_{\bar n}$ 
with respect to the relation (\ref{relation}). In particular, the tower of $W_s$ provides the following tower of open subsets by passing to the quotient space, 
\begin{displaymath}
\emptyset = \tilde{W}_{-1} \subset \tilde{W}_0 \subset \ldots \subset \tilde{W}_{k+ \sum_{q=2}^r TC(S^{n_q})} \cong P_{\bar n} \times P_{\bar n}
\end{displaymath}
and 
\begin{displaymath}
\tilde{W}_{s} - \tilde{W}_{s-1}
= \bigcup_{i+\sum_{q=2}^r j_q =s} \dfrac {F_i \times \prod_{q=2}^r G_{j_q}} {\sim}
\end{displaymath}
is a topologically disjoint union where $s=0, \ldots , k + \sum_{q=2}^rTC(S^{n_q})$. Assembling the maps $\tilde{\psi}_{(i , j_2 , \ldots , j_r)}$ we get a continuous motion planner on $\tilde{W}_{s} - \tilde{W}_{s-1}$ and, by Proposition \ref{Caracterization-TC}, we can conclude that 
$$\text{TC}(P_{\bar{n}}) \leq k+ \sum_{q=2}^{r}\text{TC}(S^{n_q})=\text{TC}(P^{n_1}) + \sum_{q=2}^{r}\text{TC}(S^{n_q}).$$
\end{proof}
\bibliographystyle{plain}

\end{document}